\documentclass{amsart}

\usepackage{amsmath}
\usepackage{amsthm}
\usepackage{enumerate}
\usepackage{microtype}
\usepackage[pdftex]{graphicx}

\DeclareMathOperator{\bd}{bd}

\newtheorem{theorem}{Theorem}
\newtheorem{lemma}{Lemma}
\newtheorem{claim}{Claim}
\begin{document}
\title[A bound for the redundant vertex Theorem on surfaces]{A single
  exponential bound for the redundant vertex Theorem on surfaces}

\author{Fr\'ed\'eric Mazoit}%
\email{Frederic.Mazoit@labri.fr}%
\thanks{This research was supported by the french ANR project DORSO.}%

\address{LaBRI, Universit\'e de Bordeaux\\
  351 cours de la libération, F-33405 Talence CEDEX, France.}


%%%%%%%%%%%%%%%%%%%%%%%%%%%%%%%%%%%%%%%%%%%%%%%%%%%%%%%%%

\begin{abstract}
  Let $s_1$, $t_1$,\dots $s_k$, $t_k$ be vertices in a graph $G$
  embedded on a surface $\Sigma$ of genus $g$.  A vertex $v$ of $G$ is
  ``redundant'' if there exist $k$ vertex disjoint paths linking $s_i$
  and $t_i$ ($1\leq i\leq k$) in $G$ if and only if such paths also
  exist in $G-v$.  Robertson and Seymour proved in Graph Minors VII
  that if $v$ is ``far'' from the vertices $s_i$ and $t_j$ and $v$ is
  surrounded in a planar part of $\Sigma$ by $l(g, k)$ disjoint
  cycles, then $v$ is redundant.  Unfortunately, their proof of the
  existence of $l(g, k)$ is not constructive.  In this paper, we give
  an explicit single exponential bound in $g$ and $k$.
\end{abstract}

\maketitle

%%%%%%%%%%%%%%%%%%%%%%%%%%%%%%%%%%%%%%%%%%%%%%%%%%%%%%%%%%%%%%%%%%%%%%%%
%%%%%%%%%%%%%%%%%%%%%%%%%%%%%%%%%%%%%%%%%%%%%%%%%%%%%%%%%%%%%%%%%%%%%%%%
%%%%%%%%%%%%%%%%%%%%%%%%%%%%%%%%%%%%%%%%%%%%%%%%%%%%%%%%%%%%%%%%%%%%%%%%
\section{Introduction}
In their graph minors series of papers, Robertson and Seymour obtained
some major results: finite graphs are
well-quasi-ordered~\cite{RoSe04b} for the minor relation, the
$k$-disjoint path problem is polynomial~\cite{RoSe95b}.  And some
notions introduced such as face-width~\cite{RoSe94a} and some
intermediate results such as the structure Theorem for graph excluding
a $K_t$ minor~\cite{RoSe03a} also proved to be of major importance.
Unfortunately, these papers are had to read and some constant are only
given in existential statements.  This has the unfortunate effect that
no algorithms given in~\cite{RoSe95b} is explicit.

Some parts such as the generalised Kuratowski Theorem for surfaces
have been rewritten and are now well understood to the point that they
appear in textbooks~\cite{Di05b}.  But some parts such as the
structure Theorem for graphs excluding a $K_t$ minor~\cite{KaWo11a} of
the unique Linkage Theorem~\cite{KaWo10a} took much longer to be workd
on.  Unfortunately, until recently, no explicit constant was known
because of a result given in~\cite{RoSe88a} whose first proofs were
only existential and which is needed in~\cite{KaWo10a,KaWo11a}: the
redundant vertex Theorem on surfaces.  In~\cite{KaWo10a}, the author
write ``At the moment, we believe that we also have a much shorter
proof of [\dots] the aspects of~\cite{RoSe88a} which we use.''  But
they seem to have never published their proof.

As already stated, the first proof~\cite{RoSe88a} is only existential.
Later, Seymour and Johnson announced a new still existential proof but
never published it.  Using ideas of this new proof, Huynh~\cite{Hu09a}
obtained a new existential proof.  This proof is still existential but
only because it lacks a topological argument which the author and some
co-authors recently filled~\cite{GeHuRi13a}.  Although their bound is
the first explicit one, it in tower of exponential in $k$ and the $g$.
In the meantime, Adler et col.~\cite{AdKoKrLoSaTh11a} proved that
$2^{k-2}\leq l(0, k)\leq \text{cst}^k$.

In this paper, we use the same approach as in~\cite{Hu09a,GeHuRi13a}
but with a more careful analysis, we prove that $l(g, k)\leq
\text{cst}^{g+k}$ nested cycles are enough to ensure that the central
vertex is redundant.

%%%%%%%%%%%%%%%%%%%%%%%%%%%%%%%%%%%%%%%%%%%%%%%%%%%%%%%%%%%%%%%%%%%%%%%%
%%%%%%%%%%%%%%%%%%%%%%%%%%%%%%%%%%%%%%%%%%%%%%%%%%%%%%%%%%%%%%%%%%%%%%%%
%%%%%%%%%%%%%%%%%%%%%%%%%%%%%%%%%%%%%%%%%%%%%%%%%%%%%%%%%%%%%%%%%%%%%%%%
\section{Statement of the Theorem}
A \emph{$k$-pattern} in a graph $G$ is a collection $\Pi := \{\{s_i,
t_i\} \;;\; 1\leq i\leq k\}$ of $k$ pairwise disjoint subsets of
$V(G)$, where each set in $\Pi$ has size one or two (i.e. $s_i$ may be
equal to $t_i$).  The \emph{vertex set} of $\Pi$ is the set
$V(\Pi):=\cup\Pi$.  A $\Pi$-linkage in $G$ is a collection
$\mathcal{L}:=\{L_1,\dots ,L_k\}$ of pairwise disjoint paths of $G$
where each $L_i$ has ends $s_i$ and $t_i$.  A vertex $v$ of $G$ is
\emph{redundant} (with respect to $\Pi$), if $G$ has a $\Pi$-linkage
if and only if $G-v$ does.  In the following, we identify
$\mathcal{L}$ with the underlying graph $\cup\mathcal{L}$.  Note that
allowing a singleton $\{s\}$ in $\Pi$ may seem strange because there
is a $\Pi$-linkage in $G$ if and only if there is a $(\Pi-s)$-linkage
in $G-s$ but we allow them for technical reasons which will become
clear later on.

A \emph{surface} is a connected compact 2-manifold possibly with
boundaries.  Oriented surfaces can be obtained by adding ``handles''
to the sphere, and non-orientable surfaces, by adding ``crosscaps'' to
the sphere.  The \emph{Euler genus} $g(\Sigma)$ of a surface $\Sigma$
(or just \emph{genus}) is twice the number of handles if $\Sigma$ is
orientable, and is the number of crosscaps otherwise.  We denote the
boundary of $\Sigma$ by $\bd(\Sigma)$.  A \emph{curve} in $\Sigma$ is
continuous function $\gamma:[0, 1]\to\Sigma$ (we identify $\gamma$
with its image), and the \emph{ends} of $\gamma$ are the points
$\gamma(0)$ and $\gamma(1)$. A \emph{path} in $\Sigma$ is an injective
curve, and a \emph{$\bd(\Sigma)$-path} is a path whose ends lie in
$\bd(\Sigma)$.  A path is \emph{contractible} if it bounds a disc, and
two paths $\mu$ and $\nu$ in $\Sigma$ are \emph{homotopic} if $\mu$
can be continuously distorted into $\nu$.  A surface with boundary
$\Omega$ with a closed disc $\Delta(\Omega)\subseteq \Omega$ is
\emph{disk with $s$ strips} if $\Omega\setminus\Delta(\Omega)$ has $s$
components called \emph{strips} which are homeomorphic to $[0,
1]\times]0, 1[$.  The \emph{ends} of a strips are the components of
the closure of $S$ minus $S$, and its \emph{sides} are its subsets
homeomorphic to $\{0\}\times]0, 1[$ and $\{1\}\times]0, 1[$.

A \emph{$k$-pattern} in a surface with boundary $\Sigma$ is a
collection $\Pi=\{\{s_i, t_i\} \;;\; 1\leq i\leq k\}$ of $k$ pairwise
disjoint subsets of $\bd(\Sigma)$, each of size one or two.  A
\emph{topological $\Pi$-linkage} is a collection $\Gamma:=\{\gamma_1,
\dots,\gamma_k\}$ of disjoint $\bd(\Sigma)$-paths in $\Sigma$ where
each $\gamma_i$ has ends $s_i$ and $t_i$. If $\Sigma$ contains a
$\Pi$-linkage, we say that $\Pi$ is \emph{topologically feasible}.
When considering a disc with strip $\Omega$, we further forbid
vertices in $V(\Pi)$ to be \emph{incident} with strip (i.e. meet the
closure of a strip).  As for $\Pi$-linkage in a graph, we abuse
notation and indentify $\Gamma$ with a the corresponding subset of
$\Sigma$.  Moreover, if $G$ is embedded in $\Sigma$ and $\Gamma$
subset of $\Sigma$ is a subgraph of $G$, we also see $\Gamma$ as a
$\Pi$-linkage in $G$.

Let $\Pi$ be a $k$-pattern in a graph $G$ embedded in a surface
$\Sigma$. A \emph{$t$-dartboard in $G$ (with respect to $\Pi$)} is
subgraph of $G$ whose components are $v$ and cycles $C_1$, \dots,
$C_t$ such that
\begin{enumerate}[i.]
\item each $C_i$ bounds a disc $\Delta_i$ in $\Sigma$;
\item $v\in\Delta_1\subset\dots\Delta_t$;
\item $V(\Pi)$ is disjoint from the interior of $\Delta_t$.
\end{enumerate}
The vertex $v$ is the \emph{centre} of the dartboard.

Our main Theorem is the following:
\begin{theorem}\label{thm:main}
  Let $l(g, k)=(20k/9)\cdot(3e^{10/(3e)})^{3(g-1)+2k}$.  The centre of
  any $l(g, k)$-dartboard with respect to a $k$-pattern $\Pi$ in a
  graph $G$ embedded on a surface $\Sigma$ of genus $g$ is redundant.
\end{theorem}

%%%%%%%%%%%%%%%%%%%%%%%%%%%%%%%%%%%%%%%%%%%%%%%%%%%%%%%%%%%%%%%%%%%%%%%%
%%%%%%%%%%%%%%%%%%%%%%%%%%%%%%%%%%%%%%%%%%%%%%%%%%%%%%%%%%%%%%%%%%%%%%%%
%%%%%%%%%%%%%%%%%%%%%%%%%%%%%%%%%%%%%%%%%%%%%%%%%%%%%%%%%%%%%%%%%%%%%%%%
\section{From the general case to reduced instances}
Our proof has two main steps.  In this Section, we prove the first
part in which we reduce the problem to a so called ``reduced
instance''.

Let $\mathcal{L}$ be a $\Pi$ linkage in a graph $G$ embedded in a
surface $\Sigma$ for some $k$-pattern $\Pi$.  Let $\mathcal{C}$ be a
$t$-dartboard with respect to $\Pi$.  A subpath $P=v_0v_1\dots v_p$ of
$\mathcal{L}$ such that there exists $0\leq i\leq t-p$ with the
property that each $v_i$ belongs to $V(C_{i+j})$ ($0\leq j\leq p$) is
\emph{increasing (from $i$ to $i+p$)}, and it is \emph{decreasing} if
the reverse path is increasing.  A \emph{valley} is a subpath
$v_{-p}\dots v_p$ of $\mathcal{L}$ with both $v_{-p}\dots v_0$ and
$v_p\dots v_0$ decreasing from $t$ to $t-p$. A \emph{bad valley} is a
valley whose vertices $v_{-p}$ and $v_p$ belong to a same end of a
strip of $\Omega$.

Let $\mathcal{L}$ be a $\Pi$ linkage in a graph $G$ embedded in a disc
with strips $\Omega$ for some $k$-pattern $\Pi$ in $G$.  We say that
$(\Omega, G, \Pi, \mathcal{L})$ is a \emph{reduced $t$-instance} if
\begin{enumerate}[i.]
\item $\Pi$ is a $k$-pattern in $\Omega$;
\item $\bd(\Omega)$ is a subgraph of $G$;
\item $G-E(\mathcal{L})$ is a $t$-dartboard $\mathcal{C}$ such that
  $\Delta_t=\Delta(\Omega)$;
\item $V(G)=V(\mathcal{C})=V(\mathcal{L})$;
\item the components of $\mathcal{L}\cap\Delta(\Omega)$ are valleys
  which are not bad.
\end{enumerate}

We prove Theorem~\ref{thm:main} as a corollary of the following
theorem.
\begin{theorem}\label{thm:strips}
  Let $f(s, k) = (20k/9)\cdot(3e^{10/(3e)})^s$.  If $(\Omega, G, \Pi,
  \mathcal{L})$ is a reduced $f(s, k)$-instance on a disc with $s$
  strips, then the centre of the dartboard $G-E(\mathcal{L})$ is
  redundant.
\end{theorem}

We now prove Theorem~\ref{thm:main} assuming Theorem~\ref{thm:strips}.

\let\save=\thetheorem \let\thetheorem=1
\begin{theorem}
  Let $l(g, k)=(20k/9)\cdot(3e^{10/(3e)})^{3(g-1)+2k}$.  The centre of
  any $l(g, k)$-dartboard with respect to a $k$-pattern $\Pi$ in a
  graph $G$ embedded on a surface $\Sigma$ of genus $g$ is redundant.
\end{theorem}
\let\thetheorem=\save
\begin{proof}
  Let $f$ be the function in the statement of
  Theorem~\ref{thm:strips}.  We claim that $l(g, k)=f(3(g-1)+2k, k)$
  satisfy the conditions of the Theorem.  Obviously, if $G$ contains
  no $\Pi$-linkage, then removing the centre of the dartboard will not
  change anything.  So we can assume that $G$ contains a $\Pi$-linkage
  $\mathcal{L}$.  Assume for a contradiction that the theorem does not
  hold for some $g$ and $k$.  Let us choose a counter example with
  $|V(G)|+|E(G)|$ minimum.  Let $\mathcal{C}$ be a $t$-dartboard with
  respect to $\Pi$ for $t=l(g, k)$, and let $v$ and $C_1$, \dots,
  $C_t$ and $v$ be its components in order.

  \begin{claim}
    $\mathcal{C}=G-E(\mathcal{L})$ and
    $V(G)=V(\mathcal{C})=V(\mathcal{L})$.
  \end{claim}
  \begin{proof}[Subproof]
    We can delete all edges not in $E(\mathcal{C})\cup
    E(\mathcal{L})$, and contract all edges in $E(\mathcal{C})\cap
    E(\mathcal{L})$.  If $x$ belongs to the symmetric difference
    $V(\mathcal{L})\Delta V(\mathcal{C})$, then we can contract any
    edge $xy$ incident with $x$ onto $y$.  All cases yield smaller
    counter-examples.
  \end{proof}
  Note that in this step we may end up identifying $s_i$ and $t_i$.
  This is the main reason why we allow $s_i=t_i$ in our $k$-patterns.

  \begin{claim}\label{claim:start}
    The first and last edge of every $L\in\mathcal{L}$ are contained
    in $\Delta_t$.
  \end{claim}
  \begin{proof}
    If not, we can move the vertex $x\in V(\Pi)$ to the other end of
    the faulty edge and remove the edge resulting in a smaller
    counter-example.
  \end{proof}

  \begin{claim}
    No edge $e$ of $\mathcal{L}\cap \Delta_t$ has both ends on the
    same cycle $C_i$.
  \end{claim}
  \begin{proof}[Subproof]
    If not, let $C'_i$ be the cycle of $C_i\cup\{e\}$ which contains
    $e$ and which bounds a disc containing $v$.  Replacing $C_i$ with
    $C'_i$ and removing all the edges in $E(C_i)\setminus E(C'_i)$
    yields a smaller counter example.
  \end{proof}

  \begin{claim}\label{claim:hill}
    Every component of $\mathcal{L}\cap\Delta_t$ is a valley.
  \end{claim}
  \begin{proof}[Subproof]
    If not, there exist a subpath $P=v_{-p}\dots v_p$ of some
    $L\in\mathcal{L}$ such that $v_{-p}\dots v_0$ and $v_p\dots v_0$
    increases from $i$ to $i+p$.  Such a path is a \emph{hill from
      $i$}.  Among all hills choose one from a minimum $i$.  Let $Q$
    be the subpath of $C_i$ such that $P\cup Q$ bounds a disc
    $\Delta_P$ which does not contain $v$.  No $L'\in\mathcal{L}$
    crosses $Q$.  Indeed, since $\mathcal{L}$ contain disjoint paths,
    any such $L'$ would then contain a hill from $i-1$ which
    contradict the choice of $i$.  We can thus replace $P$ in
    $\mathcal{L}$ by $Q$ and remove all the edges in $P$ to obtain a
    smaller counter example.
  \end{proof}

  So far we focused only on edge inside $\Delta_t$.  Let us study
  edges outside $\Delta_t$.  Every such edge $xy$ together with the
  radius $vx$ and $vy$ define a closed curve $\mu_{xy}$ and any two
  such curves are disjoint except from the base point $v$.  An edge
  $e$ is \emph{non contractible} if $\mu_e$ is and two edges $e$ and
  $f$ are \emph{homotopic} if $\mu_e$ and $\mu_f$ are.

  \begin{claim}\label{claim:1}
    For every edge $e$ outside $\Delta_t$ such that $\mu_e$ bounds a
    disc $\Delta_e$, there exists a vertex of $V(\Pi)$ in the interior
    of $\Delta_e$.
  \end{claim}
  \begin{proof}[Subproof]
    If not, let $\Delta'_t=\Delta_t\cup\Delta_e$ and $C'_t$ be the
    boundary of $\Delta'_t$.  Replacing $C_t$ with $C'_t$ and removing
    the edges in $E(C_t)\setminus E(C'_t)$ yields a smaller
    counter-example.
  \end{proof}

  Note that if $\mu_{xy}$ bounds a disc $\Delta_{xy}$ and $\mu_{x'y'}$
  bound a disc $\Delta_{x'y'}$, then either $\Delta_{xy}$ and
  $\Delta_{x'y'}$ are disjoint (except from $v$) or one contains the
  other.  Let $E_c$ contain the edges such that the corresponding
  discs are maximal.
  \begin{claim}
    The set $E_c$ contains at most $2k$ edges.
  \end{claim}
  \begin{proof}[Subproof]
    This follows from the fact that, because of Claim~\ref{claim:1},
    each such maximal disc must contain at least one vertex of
    $V(\Pi)$.
  \end{proof}

  \begin{claim}
    There are at most $3(g-1)$ homotopy classes of non contractible
    edges.
  \end{claim}
  \begin{proof}[Subproof]
    Let $\mathcal{D}$ contain one curve $\mu_e$ from each homotopy
    class of non contractible edge.  The curves in $\mathcal{D}$ are
    the edges of a loopless simple graph embedded on $\Sigma$ with one
    vertex $v$.  Let $H$ be such a graph maximal with the property
    that $\mathcal{D}\subseteq E(H)$.  Every face of $H$ is then a
    disc which is bounded by exactly 3 edges.  Thus,
    $3|F(H)|=2|E(H)|$.  When combining this equality and Euler's
    formula (i.e. $|V(H)|+|F(H)|= |E(H)|+2-g$), we obtain that
    $3+2|E(H)|\geq 3|E(H)|+6-3g$, and thus $|\mathcal{D}|\leq
    |E(H)|\leq 3g-3$.
  \end{proof}

  For each homotopy class $\mathcal{E}$ of non contractible edges, we
  can choose a strip in $\Sigma$ whose sides belong to $\mathcal{E}$
  and which contains $\mathcal{E}$ and no other edge of $G$.  For each
  edge $e\in E_c$ if we remove from $\Delta_e\setminus \Delta_t$ a
  small disc around a vertex of $V(\Pi)$, we also obtain a strip.  In
  this way, we can embed $G$ on a first disc with at most
  $3(g-1)+|E_c|$ strips $\Omega_1$.

  Unfortunately, although we know that at least $|E_c|$ vertices of
  $V(\Pi)$ are not incident with strips of $\Omega_1$, some other
  vertex $u\in V(\Pi)$ may be.  But then, because of
  Claim~\ref{claim:start}, there is a face $F$ of $G$ contained in
  some strip $S$ such that $u$ is incident with no strip of
  $\Omega_1\setminus F$.  Removing $F$ from $\Omega_1$ ``splits'' $S$
  in two, and after at most $2k-|E_c|$ such splitting, we obtain a
  disc with at most $3(g-1)+2k$ strips $\Omega$ with $V(\Pi)$ being
  incident no strip of $\Omega$.  Thus $(\Omega, G, \Pi, \mathcal{L})$
  is almost a reduced $t$-instance.  The only remaining problem is
  that there could be bad valleys.

  \begin{claim}
    The instance $(\Omega, G, \Pi, \mathcal{L})$ is a reduced
    $t$-instance.
  \end{claim}
  \begin{proof}[Subproof]
    Suppose that $P=v_{-p}\dots v_p$ is a valley and $v_{-p}$ and
    $v_p$ both belong to a same end of a strip $S$.  By following $P$
    along $\mathcal{L}$ in both directions, we obtain a path
    $P'=v_{-p-2}\dots v_{p+2}$ such that $v_{-p-2}v_{-p-1}$ and
    $v_{p+2}v_{p+1}$ are edge from $C_{t-1}$ to $C_t$ and
    $v_{-p-1}v_{-p}$ and $v_{p+1}v_p$ cross $S$ in the same direction.
    The path $P'$ looks like a hill as defined in
    Claim~\ref{claim:hill} except that the hill is so high that it
    ``traverses the sky'' through $S$.  We thus define \emph{mountains
      from $i$} as a subpath $v_{-p-1-i}\dots v_{p+1+i}$ of
    $\mathcal{L}$ such that
    \begin{itemize}
    \item $v_{-p-1-i}\dots v_{-p-1}$ and $v_{p+1+i}\dots v_{p+1}$
      increases from $t-i$ to $t$,
    \item and $v_{-p-1}v_{-p}$ and $v_{p+1}v_p$ cross a strip $S$;
    \item $v_{-p}\dots v_p$ is a bad hill.
    \end{itemize}

    What we just proved is that if there is a bad hill, then there
    exists a mountain.  As in the proof of Claim~\ref{claim:hill}, we
    can easily shortcut a mountain from some minimal $i$ and obtain a
    smaller counter-example.
  \end{proof}

  We have now finished our cleaning process and we can apply
  Theorem~\ref{thm:strips}.  Indeed, $(\Omega, G, \Pi, \mathcal{L})$
  is a reduced $f(3(g-1)+2k, k)$-instance on a disc with at most
  $(3(g-1)+2k)$-strips.  The centre $v$ is thus redundant.
\end{proof}

%%%%%%%%%%%%%%%%%%%%%%%%%%%%%%%%%%%%%%%%%%%%%%%%%%%%%%%%%%%%%%%%%%%%%%%%
%%%%%%%%%%%%%%%%%%%%%%%%%%%%%%%%%%%%%%%%%%%%%%%%%%%%%%%%%%%%%%%%%%%%%%%%
%%%%%%%%%%%%%%%%%%%%%%%%%%%%%%%%%%%%%%%%%%%%%%%%%%%%%%%%%%%%%%%%%%%%%%%%
\section{The proof for reduced instances}

The strategy to prove Theorem~\ref{thm:strips} is the following.  Let
$(\Omega, G, \Pi, \mathcal{L})$ be a reduced $t$-instance with $t$
large.  We first find a topological $\Pi$-linkage $\Gamma$ which
crosses the strip ``few'' times (using Theorem~\ref{thm:topo}).  The
idea is to try to realise this topological linkage in $G$.  To do so,
two cases arise.  If $\mathcal{L}$ crosses all the strips ``enough''
time, then we explicitly realise $\Gamma$ in $G$.  Otherwise, we can
cut ``small'' strips and reduce to a $k+\text{``few''}$ disjoint path
problem on a disc with less strips.

Our tool to find a good topological linkage is the following Theorem
of Geelen et col.~\cite{GeHuRi13a}.
\begin{theorem}\label{thm:topo}
  Let $\Sigma$ be a surface with boundary, $\Pi$ be a topologically
  feasible $k$-pattern in $\Sigma$, and $P$ be a non-separating
  $\bd(\Sigma)$-path in $\Sigma$ whose ends are disjoint from
  $V(\Pi)$.  There exists a $\Pi$-linkage $\Gamma$ in $\Sigma$ such
  that each path $\gamma\in\Gamma$ intersects $P$ at most twice.
\end{theorem}

We also need the two following easy Lemmas.
\begin{lemma}\label{lem:base-case}
  Let $\Pi$ be a $k$-pattern in a graph $G$ embedded on a disc
  $\Delta$.  Let $\mathcal{L}$ be a $\Pi$-linkage in $G$.  If
  $(\Delta, G, \Pi, \mathcal{L})$ is a reduced $t$-instance for $t\geq
  k$, then $v$ is redundant.
\end{lemma}
\begin{proof}
  The proof is by induction on $k$.  A \emph{border path in
    $\mathcal{L}$} is an $L\in\mathcal{L}$ such that one of the
  component of $\Delta\setminus L$ contains no path in $\mathcal{L}$.
  Any such a path which meets $C_{t-1}$ can always be rerouted to a
  subpath of $C_t$ linking its ends.  Let thus assume that no border
  path in $\mathcal{L}$ meets $C_{t-1}$.  Since all $\Pi$-linkage in a
  disc have at least one border path, then
  $\mathcal{L}'=\mathcal{L}\cap\Delta_{t-1}$ is a $\Pi'$ linkage in
  $G'=G\cap \Delta_{t-1}$ for some $k'$-pattern $\Pi'$ with $k'<k$.
  Note that $(\Delta_{t-1}, G', \Pi', \mathcal{L}')$ is a reduced
  instance and $t-1>k'$.  There thus exists a $\Pi'$-linkage
  $\mathcal{L}''$ which avoids $v$ in $G'$.  But then
  $\mathcal{L}''\cup(\mathcal{L}\setminus\Delta_{t-1})$ is a
  $\Pi$-linkage in $G$ which avoids $v$.
\end{proof}

The following Lemma can be proved in a very similar way.  We thus
leave the proof to the reader.
\begin{lemma}\label{lem:cylinder}
  Let $G=P_k\times C_n$ be a cylinder embedded on a disc $\Delta$ such
  that $\bd(\Delta)$ is a cycle $C$ of $G$, and let $\Pi$ be a
  $k$-pattern in $G$ with $V(\Pi)\subseteq V(C)$ (and thus $n\geq
  2k$).  If $\Pi$ is topologically feasible, then $G$ contains a
  $\Pi$-linkage.
\end{lemma}

Let $(\Omega, G, \Pi, \mathcal{L})$ be a reduced instance.  The
\emph{size} $|S|$ of a strip $S$ of $\Omega$ is the number $p$ of
components of $S\cap\mathcal{L}$.  We say that $\mathcal{L}$
\emph{crosses} $S$ $p$ times.

\begin{lemma}\label{lem:tool}
  Let $(\Omega, G, \Pi, \mathcal{L})$ be a reduced $t$-instance on a
  disc with $s$ strips for $t\geq 2k3^s$.  If for each $1\leq i\leq
  s$, $|S_i|\geq 3k3^i+1$, then the centre $v_\text{centre}$ of the
  dartboard $G\setminus \mathcal{L}$ is redundant.
\end{lemma}
\begin{proof}
  We prove the Lemma by induction on the lexicographic order on $(s,
  k)$.  Suppose that the set $\mathcal{L}_0$ of the singletons of
  $\mathcal{L}$ is nonempty.  Remove $E(C_t)$ from $G$.  As long as
  some degree 2 vertex $u$, contract an edge incident with $u$ and
  then remove from $\Omega$ the faces incident with $\bd(\Delta_t)$.
  We then obtain reduced $(t-1)$-instance for a $(<k)$-pattern $\Pi'$
  in a graph $G'$ embedded on a disc with $s$ strips.  Since
  $(k-1)3^s<k3^s-1$, then $v_\text{centre}$ is redundant.  There thus
  exists a $\Pi'$ linkage $\mathcal{L}'$ in $G'$.  But the
  $\mathcal{L'}\cup(\mathcal{L}\setminus\Delta_{t-1})\cup\mathcal{L}_0$
  is a $\Pi$-linkage in $G$ which avoids $v_\text{centre}$.  \medskip

  For $1\leq i\leq s$, let $\Omega_i=\Omega\setminus(S_1\cup\dots\cup
  S_i)$.  Note that $\Omega_0=\Omega$ and $\Omega_s=\Delta_t$.
  Because $\mathcal{L}$ contains no bad valley, on each end of $S_i$,
  the middle $k3^i+1$ paths have to either ``go over'' the $k3^i$
  paths on their left or on their right.  They thus contains subpaths
  going from $C_{t-k3^i+1}$ through $S_i$ and then down to
  $C_{t-k3^i+1}$.  Let $\mathcal{P}_i$ contain these $k3^i+1$
  subpaths, and let $\mathcal{Q}_i$ contain the middle $2k3^{i-1}$
  paths of $\mathcal{P}_i$ (i.e. $\mathcal{Q}_i$ leaves $\lceil
  k3^{i-1}/2\rceil$ paths of $\mathcal{P}_i$ on one side and $\lfloor
  k3^{i-1}/2 \rfloor +1$ ones on the other side).

  \begin{claim}\label{claim:topo}
    There exists a topological linkage $\Gamma$ such that for $0\leq
    i\leq s$, $\Gamma$ crosses $S_i$ at most $2k3^{i-1}$.  Moreover,
    we can suppose that in each strip $S_i$, $\Gamma\cap S_i$ is a
    subgraph of $\mathcal{Q}_i\cap S_i$.
  \end{claim}
  \begin{proof}[Subproof.]
    We prove by induction on $0\leq i \leq s$ that there exists a
    topological $\Pi$-linkage $\Gamma_i$ such that for $1\leq j\leq
    i$, $\Gamma_i$ crosses each strip $S_j$ at most $2k3^{j-1}$.  For
    $i=0$, we set $\Gamma_0=\mathcal{L}$.  So suppose that
    $\Gamma_{i-1}$ has been defined for $1\leq i\leq s$.  By
    construction, $\Gamma_{i-1}\cap\Omega_{i-1}$ is a
    $\Pi_{i-1}$-linkage for some $w$-pattern $\Pi_{i-1}$ such that
    $w\leq k+2k3^{1-1}+2k3^{2-1}+\dots+2k3^{(i-1)-1}=k3^{i-1}$.  By
    Theorem~\ref{thm:topo}, there exists a $\Pi_{i-1}$-linkage
    $\Gamma'_i$ which crosses $S_i$ at most $2w\leq 2k3^{i-1}$ times,
    and since $2k3^{i-1}=|\mathcal{Q}_i|$, we can suppose that
    $\Gamma'_i\cap S_i$ is a subgraph of $\mathcal{Q}_i\cap S_i$.  But
    then $\Gamma_i=\Gamma'_i\cup(\Gamma_{i-1}\setminus\Omega_{i-1})$
    is a $\Pi$-linkage in $\Omega$ which satisfies the induction
    conditions for $i$.  The linkage $\Gamma=\Gamma_s$ satisfies the
    conditions of the Claim.
  \end{proof}

  In the remaining of this proof, we show that we can indeed realise
  in $G$ the topological linkage $\Gamma$ given by the previous Claim.
  To do so, we first fix an orientation of $\Delta_t$ so that we can
  order elements of $\bd(\Delta_t)$ from left to right.  As in the
  proof of Claim~\ref{claim:topo} for $0\leq i\leq s$,
  $\Gamma_i:=\Gamma\cap\Omega_i$ is a $\Pi_i$-linkage for some $(\leq
  k3^i)$-pattern $\Pi_i$.  A \emph{boundary segment} $l$ of $\Omega_i$
  is a component of $\bd(\Delta(\Omega_i))\cap\bd(\Omega_i)$.  Let
  $x_0$, \dots, $x_p$ be the vertices of $V(\Pi_i)$ in a boundary
  segment $\alpha$ of $\Omega$ in order.  An \emph{$\alpha$-pyramid}
  is a set of paths $\mathcal{M}_l=\{M_0, \dots, M_p\}$ such that for
  $0\leq m\leq p/2$, $M_m$ and $M_{p-m}$ respectively link $x_m$ and
  $x_{p-m}$ and $C_{t-m}$.

  \begin{claim}
    Let $\alpha_1$ and $\alpha_2$ be the two boundary paths of
    $\Omega_{s-1}$.  There exists in $G\cap\Delta_t$ disjoint
    $\alpha_i$-pyramids which do not meet $\mathcal{P}_s$.
  \end{claim}
  \begin{proof}[Subproof.]
    We prove by induction on $0\leq i< s$ that there exists a set
    $\mathcal{M}_i$ of disjoint path such that for each boundary
    segment $\alpha$ of $\Omega_i$, $\mathcal{M}_i$ contains an
    $\alpha$-pyramid, and $\mathcal{M}_i$ is disjoint from the sets
    $\mathcal{P}_j$ for $i<j\leq s$.

    The existence of $\mathcal{M}_0$ follows from the fact that
    $\mathcal{L}$ is a $\Pi$ linkage and $\Pi$ contains no singleton.
    Indeed, since $\Pi$ contains no singleton, no path
    $L\in\mathcal{L}\cap\Delta_t$ has both ends on the same boundary
    segment $\alpha$ (otherwise $L$ and $\alpha$ would bound a disc
    which would contain a singleton in $\Pi$).  Let $\mathcal{M}_0$
    contain the paths in $\mathcal{L}\cap\Delta_t$ with an end in
    $V(\Pi)$, let $\alpha$ be a boundary segment of $\Omega_0$, and
    let $x_0$, \dots, $x_p$ be the vertices of $V(\Pi)\cap \alpha$ in
    order.  The path $L_i$ leaving $x_i$ either has to go ``over'' the
    $i$ path leaving $x_0$, \dots, $x_{i-1}$ or it has to to go
    ``over'' the $p-i$ path leaving $x_{i+1}$, \dots, $x_p$.  It thus
    as to meet $C_{t-\min(i, p-i)}$.  We can thus replace $L_i$ by a
    subpath so that $\mathcal{M}_0$ satisfies the required property.
    Note that $L$ may produce a subpath for each of its ends but these
    two subpaths do no meet because $L$ has to ``go over'' the paths
    in some $\mathcal{P}_i$.

    Suppose now that $\mathcal{M}_{i-1}$ exists for $1\leq i<s$.  The
    boundary segments of $\Omega_i$ are precisely boundary segments of
    $\Omega_{i-1}$ which are not incident with $S_i$ and the two
    unions of the ends of $S_i$ with the boundary segment of
    $\Omega_{i-1}$ to which they are incident to.  For a boundary
    segment $\alpha$ of $\Omega_i$ which is also a boundary segment of
    $\Omega_{i-1}$, we put in $\mathcal{M}_i$ the paths of the
    $\alpha$-pyramid in $\mathcal{M}_{i-1}$.  So let
    $\alpha_\text{left}\cup\beta\cup \alpha_\text{right}$ be a
    boundary segment of $\Omega_i$ in which $\beta$ is an end of $S_i$
    and $\alpha_\text{left}$ and $\alpha_\text{right}$ are boundary
    segments of $\Omega_{i-1}$ incident with $\beta$ which are
    respectively on the left and on the right of $\beta$.  Let $x_0$,
    \dots, $x_{p-1}$, $y_0$, \dots, $y_{q-1}$, $z_0$, \dots, $z_{r-1}$
    be the vertices of $\Pi_i$ on $\alpha_\text{left}\cup\beta\cup
    \alpha_\text{right}$ in order with all $x_j\in\alpha_\text{left}$,
    $y_j\in\beta$ and $z_j\in\alpha_\text{right}$.  Let
    $\widetilde{\mathcal{P}}$ contain the paths of
    $\mathcal{P}_i\setminus S_i$ with an end in $\beta$.  Let
    $\mathcal{P}_\text{middle}$ contains the paths of
    $\widetilde{\mathcal{P}}\cap (\mathcal{Q}_i\setminus S_i)$, and
    $\mathcal{P}_\text{left}$ and $\mathcal{P}_\text{right}$ contain
    respectively the paths on the left and on the right of
    $\mathcal{P}_\text{middle}$.  One of $\mathcal{P}_\text{left}$ and
    $\mathcal{P}_\text{right}$ contains $\lceil k3^{i-1}/2\rceil$
    paths and the other contains $\lfloor k3^{i-1}/2\rfloor+1$ paths.
    So Both contain at least $\lceil k3^{i-1}/2\rceil$ paths.

    Let $F_0$, \dots, $F_{p-1}$ be the paths of the
    $\alpha_\text{left}$-pyramid in $\mathcal{M}_i$ taken from left to
    right, and let $H_0$, \dots, $H_{r-1}$ be the paths of the
    $\alpha_\text{right}$-pyramid in $\mathcal{M}_i$ taken from left
    to right.  Let $\widetilde{\mathcal{M}}$ be the set of paths which
    contains
    \begin{enumerate}[i.]
    \item the paths $F_j$ for $0\leq j\leq (p-1)/2$,
    \item\label{enum:2} the paths obtained by following $F_{p-1-j}$,
      then going right along $C_{t-j}$ to the $j^\text{th}$ path of
      $\mathcal{P}_\text{left}$ on the right and then along this path
      down to $C_{t-k3^i+1}$ for $0\leq j<(p-1)/2$,
    \item\label{enum:3} the paths in $\mathcal{P}_\text{middle}$ with
      $y_j$ as an end ($0\leq j<q$),
    \item\label{enum:4} the paths obtained by following $H_{r-1-j}$,
      then going left along $C_{t-j}$ to the $j^\text{th}$ path of
      $\mathcal{P}_\text{right}$ on the left and then along this path
      down to $C_{t-k3^i+1}$ for $0\leq j<(p-1)/2$,
    \item the path $H_{r-1-j}$ for $0\leq j\leq (r-1)/2$.
    \end{enumerate}
    Note that these paths are well defined because, $\Pi_{i-1}$ is a
    $k3^{i-1}$-pattern, which implies that $p+r\leq k3^{i-1}$, and
    thus we respectively only send $\lfloor p/2\rfloor\leq
    |\mathcal{P}_\text{left}|$ paths in $\mathcal{P}_\text{left}$ and
    $\lfloor r/2\rfloor \leq |\mathcal{P}_\text{right}|$ paths in
    $\mathcal{P}_\text{right}$.  Now, $\widetilde{\mathcal{M}}$
    contains $p+q+r\leq k3^i$ paths, and since in the cases
    \ref{enum:2}, \ref{enum:3} and \ref{enum:4} the path reach
    $C_{t-k3^i+1}$, we can thus obtain a
    $\alpha_\text{left}\cup\beta\cup\alpha_\text{right}$-pyramid by
    shortening paths in $\widetilde{\mathcal{M}}$.  We put the paths
    of this pyramid in $\mathcal{M}_i$.  Since the paths in
    $\mathcal{M}_{i-1}$ were disjoints from the paths in
    $\mathcal{P}_j$ for $j>i$, then by construction, $\mathcal{M}_i$
    contains disjoint paths which are disjoint from the paths in
    $\mathcal{P}_j$ for $j>i$.  This finishes the induction.
  \end{proof}

  To find the cylinder, we proceed almost as in the pyramid building
  induction step.  Let $\alpha_1$ and $\alpha_2$ be the two boundary
  segments of $\Omega_{s-1}$ and let $\beta_1$ and $\beta_2$ be the
  ends of $S_s$ such that $\alpha_1$ is on the left of $\beta_1$.  Let
  $\mathcal{P}_{\beta_i}$ be the paths of $\mathcal{P}_s\setminus S_s$
  with an end in $\beta_i$ ($i=1$, $2$).  We extend the paths in the
  pyramids by sending, for $i=1$, $2$,
  \begin{itemize}
  \item the path on the left of the $\alpha_i$-pyramid to the left and
    use at most $\lceil k3^s/2\rceil$ paths on the right of
    $\mathcal{P}_{\beta_{2-i}}$ to reach $C_{t-k3^s+1}$;

  \item the path on the right of the $\alpha_i$-pyramid to the right
    and use at most $\lceil k3^s/2\rceil$ paths on the left of
    $\mathcal{P}_{\beta_i}$ to reach $C_{t-k3^s+1}$.
  \end{itemize}
  Note that there rerouting are possible because there are $\lceil
  k3^s/2\rceil$ paths or $\lfloor k3^s/2\rfloor+1\geq \lceil
  k3^s/2\rceil$ paths free to accommodate them.  In the end, we have
  found a set $\mathcal{M}$ of disjoint path in $G\cap\Delta_t$
  linking the vertices of $V(\Pi_s)$ to $C_{t-k3^s+1}$.  Since $\Pi_s$
  is a $k3^s$-pattern, Lemma~\ref{lem:cylinder} implies that there is
  a $\Pi_s$-linkage $\mathcal{L}'$ in $\mathcal{P}\cup
  (C_{t}\cup\dots\cup C_{t-k3^s+1})$.  But then
  $\mathcal{L}'\cup(\mathcal{L}\setminus\Delta_t)$ is a $\Pi$-linkage
  in $G$ which avoids $v_\text{centre}$.
\end{proof}

We now finish the proof of the Theorem.  \let\save=\thetheorem
\let\thetheorem=2
\begin{theorem}
  Let $f(s, k) = (20k/9)\cdot(3e^{10/(3e)})^s$.  If $(\Omega, G, \Pi,
  \mathcal{L})$ is a reduced $f(s, k)$-instance on a disc with $s$
  strips, then the centre of the dartboard $G-E(\mathcal{L})$ is
  redundant.
\end{theorem}
\let\thetheorem=\save
\begin{proof}
  Let us order the strips $S_1$, \dots, $S_s$ by increasing size, and
  let $|S_i|$ be the size of $S_i$.  If we can apply
  Lemma~\ref{lem:tool}, then we know that we only need $2k3^s\leq f(s,
  k)$ cycles to realise $\Pi$.  If not, then there exists a strip $i$
  such that $|S_i|\leq 3k3^i$.  Let $i_1$ be the maximum such $i$.
  Then $\mathcal{L}\cap\Omega_{i_1}$ is a $\Pi'$-linkage for some
  $k+|S_1|+\dots+|S_{i_1}|$ pattern $\Pi'$.  We know that $|S_1|\leq
  |S_{i_1}|\leq 3k3^{i_1}$.  So $\Pi'$ is a $(\leq k+3k i_1 3^{i_1})$
  pattern in $\Omega_{i_1}$.  We can then try to apply
  Lemma~\ref{lem:tool}.  If we can, then we know that $2(k+3k i_1
  3^{i_1})3^{s-i_1}$ cycles are enough to realise $\Pi'$ in $G\cap
  \Omega_{i_1}$, and thus to realise $\Pi$ in $G$.  If not, then there
  exists $i_2>0$ such that $|S_{i_1+i_2}|\leq 3(k+3k i_1
  3^{i_1})3^{i_2}$, and we can iterate.

  More formally, let $\xi=10/3$.  We recursively define $i_l>0$ to be
  the maximum index such that $|S_{i_1+\dots+i_l}| \leq k\cdot
  \xi^l\cdot i_1\cdot i_2\cdots i_{l-1}3^{i_1+\dots+i_l}$.  We can
  suppose that $i_1$ exists.  Indeed, as already noted, we can suppose
  that we cannot directly apply Lemma~\ref{lem:tool} so there exists
  $i$ such that $|S_i|\leq 3k3^i$.  But then surely $|S_i|\leq
  k\cdot\xi\cdot i\cdot 3^i$.  So suppose that $i_l$ exist and that
  $i_{l+1}$ does not.  Then $\mathcal{L}\cap\Omega_{i_1+\dots+i_l}$ is
  $\Pi'$-pattern for some $k'$-pattern $\Pi'$.  Let us now bound $k'$.
  \begin{align*}
    k'&=k+(|S_1|+\dots+|S_{i_1}|)+(|S_{i_1+1}|+\dots+|S_{i_1+i_2}|)+\dots\\
    &\shoveright{ +(|S_{i_1+\dots+i_{l-1}+1}|+\dots+|S_{i_1+\dots+i_l}|)}\\
    \displaybreak[0]
    &\leq k+i_1|S_{i_1}|+i_2|S_{i_1+i_2}|+\dots+i_l|S_{i_1+\dots+i_l}|\\
    \displaybreak[0]
    &\leq k+ k\cdot \xi^1\cdot  i_1\cdot 3^{i_1}\\
    &\qquad + k\cdot \xi^2\cdot i_1\cdot i_2\cdot3^{i_1+i_2}+\dots\\
    &\qquad + k\cdot \xi^l\cdot i_1\cdot i_2\cdots i_l\cdot 3^{i_1+\dots+i_l}\\
    \displaybreak[0]
    &\leq k\cdot \xi^l\cdot i_1\cdot i_2\cdots i_l\cdot
    3^{i_1+\dots+i_l}%
    \Bigl(1+\frac{1}{\xi\cdot i_l\cdot 3^{i_l}}\\
    &\qquad +\frac{1}{\xi^2\cdot i_li_{l-1}\cdot
      3^{i_{l-1}+i_l}}+\dots
    +\frac{1}{\xi^li_1\cdot i_2\cdots i_l3^{i_1+\dots+i_l}}\Bigr)\\
    \displaybreak[0]
    &\leq k\cdot \xi^l\cdot i_1\cdot i_2\cdots i_l\cdot
    3^{i_1+\dots+i_l}\Bigl(1+\frac{1}{3\xi}+\frac{1}{(3\xi)^2}+\dots
    +\frac{1}{(3\xi)^l}\Bigr)\\
    \displaybreak[0]
    &\leq k\cdot \xi^l\cdot i_1\cdot i_2\cdots i_l\cdot
    3^{i_1+\dots+i_l}
    \frac{1}{(3\xi)^l}\frac{(3\xi)^{l+1}-1}{3\xi-1}\\
    \displaybreak[0]
    &\leq k\cdot\frac{1}{3^l}\frac{(3\xi)^{l+1}-1}{3\xi-1} \cdot
    i_1\cdot i_2\cdots i_l\cdot 3^{i_1+\dots+i_l}
  \end{align*}

  But then, because $i_{l+1}$ does not exist, for every $j>0$,
  \begin{align*}
    |S_{i_1+\dots+i_l+j}|%
    &> k\cdot \xi^{l+1}\cdot i_1\cdot i_2\cdots
    i_l3^{i_1+\dots+i_l+j}\\
    \displaybreak[0]
    &> \xi^{l+1}3^l\frac{3\xi-1}{(3\xi)^{l+1}-1}%
    \left( k\cdot\frac{1}{3^l}\frac{(3\xi)^{l+1}-1}{3\xi-1}
      \cdot i_1\cdot i_2\cdots i_l\cdot 3^{i_1+\dots+i_l}\right)3^j\\
    \displaybreak[0]
    &> \frac{(3\xi)^{l+1}}{(3\xi)^{l+1}-1} \frac{3\xi-1}{3} k'3^j\\
    \displaybreak[0]
    &> \frac{3\xi-1}{3} k'3^j=3k'3^j
  \end{align*}
  We can thus apply Lemma~\ref{lem:tool}.  Let $s'=s-(i_1+\dots+i_l)$.
  The number of cycles need to realise $\Gamma$ is bounded by
  \begin{align*}
    C=2k'3^{s'}&\leq 2k\cdot\frac{1}{3^l}\frac{(3\xi)^{l+1}-1}{3\xi-1}
    \cdot i_1\cdot i_2\cdots i_l\cdot 3^{i_1+\dots+i_l}3^{s'}\\
    \displaybreak[0]
    &\leq 2k\cdot\xi^{l+1}
    \frac{3}{3\xi-1}\frac{(3\xi)^{l+1}-1}{(3\xi)^{l+1}}
    \cdot i_1\cdot i_2\cdots i_l\cdot 3^{i_1+\dots+i_l}3^{s'}\\
    \displaybreak[0]
    &\leq 2k\cdot\xi \cdot\frac{3}{3\xi-1} \cdot\xi^l\cdot i_1\cdot i_2\cdots i_l\cdot 3^s\\
    \displaybreak[0]
    &\leq \frac{20}{9}k{\left(\xi\frac{i_1+\cdots+i_l}{l}\right)}^l3^s
    &\qquad\qquad\llap{by convexity}\\
    \displaybreak[0]
    &\leq \frac{20}{9}k{\left(\frac{\xi s}{l}\right)}^l3^s &
    \llap{because $i_1+\cdots+i_l\leq s$}
  \end{align*}
  
  But ${\bigl(\xi s/l\bigr)}^l$ is maximum for $l=\xi s/e$.  So $C\leq
  (20k/9)\cdot e^{10s/(3e)}\cdot 3^s=f(s, k)$.  The central vertex
  $v_\text{centre}$ is thus redundant as claimed.
\end{proof}

\newcommand{\etalchar}[1]{$^{#1}$}

\end{document}